\documentclass[leqno]{amsart}

\usepackage{stmaryrd,graphicx}
\usepackage{amssymb,mathrsfs,amsmath,amscd,amsthm,color}
\usepackage{float}
\usepackage[all,cmtip]{xy}
\DeclareMathAlphabet{\mathpzc}{OT1}{pzc}{m}{it}
\usepackage{amsfonts,latexsym,wasysym}
\newcommand{\ui}{[0,1]}

\newcommand{\tXh}{\widetilde{X}_{H}}
\newcommand{\hX}{\widehat{X}}
\newcommand{\hx}{\widehat{x}}

\newcommand{\txh}{\tilde{x}_{H}}
\newcommand{\tX}{\widetilde{X}}

\newcommand{\pxxo}{P(X,x_0)}
\newcommand{\pionex}{\pi_{1}(X,x_0)}

\newcommand{\bpp}{b_{0}^{+}}
\newcommand{\lb}{\langle}
\newcommand{\rb}{\rangle}

\newcommand{\mcc}{\mathcal{C}}

\newcommand{\mci}{\mathcal{I}}

\newcommand{\scrl}{\mathscr{L}}

\newcommand{\bba}{\mathbb{A}}
\newcommand{\bbw}{\mathbb{W}}

\newcommand{\bbd}{\mathbb{D}}

\newcommand{\bbh}{\mathbb{H}}
\newcommand{\bbn}{\mathbb{N}}
\newcommand{\bbq}{\mathbb{Q}}
\newcommand{\bbr}{\mathbb{R}}

\newcommand{\bbt}{\mathbb{T}}
\newcommand{\bbz}{\mathbb{Z}}

\newcommand{\bbhp}{\mathbb{H}^{+}}

\DeclareMathOperator{\supp}{supp}
\DeclareMathOperator{\cl}{Cl}

\DeclareMathOperator{\sct}{Sc}
\newcommand{\pionet}{\pi_1(\bbt,t_0)}
\newtheorem{theorem}{Theorem}[section]
\newtheorem{lemma}[theorem]{Lemma}
\newtheorem{proposition}[theorem]{Proposition}
\newtheorem{corollary}[theorem]{Corollary}
\theoremstyle{definition}\newtheorem{definition}[theorem]{Definition}
\newtheorem{example}[theorem]{Example}

\newtheorem{remark}[theorem]{Remark}

\begin{document}
\title{Dense products in fundamental groupoids}
\author{Jeremy Brazas}
\date{\today}

\begin{abstract}
Infinitary operations, such as products indexed by countably infinite linear orders, arise naturally in the context of fundamental groups and groupoids. Despite the fact that the usual binary operation of the fundamental group determines the operation of the fundamental groupoid, we show that, for a locally path-connected metric space, the well-definedness of countable dense products in the fundamental group need not imply the well-definedness of countable dense products in the fundamental groupoid. Additionally, we show the fundamental groupoid $\Pi_1(X)$ has well-defined dense products if and only if $X$ admits a generalized universal covering space.
\end{abstract}

\maketitle

\section{Introduction}

Groups and groupoids with infinite product operations arise naturally in fundamental groups $\pi_1(X,x)$ and fundamental groupoids $\Pi_1(X)$ of a space $X$ with non-trivial local homotopy. In particular, since the components of an open set in $\ui$ may have any countable linear order type, including the dense order type of $\bbq$, it is possible to form transfinite path-concatenations $\prod_{j\in \scrl}\alpha_j$ indexed by any countable linear order $\scrl$ and hence a \textit{transfinite $\Pi_1$-product} $\prod_{j\in\scrl}[\alpha_j]:=\left[\prod_{j\in \scrl}\alpha_j\right]$ in $\Pi_1(X)$ when the concatenation is defined and continuous. If each $\alpha_j$ is a loop based at $x\in X$, the result is a \textit{transfinite $\pi_1$-product} in $\pi_1(X,x)$. However, these partial infinitary operations on homotopy classes are only well-defined if path-homotopic factors result in path-homotopic concatenations; this is not guaranteed even for some subspaces of $\bbr^3$. 

In \cite{BrazScattered}, it is shown that, for both the fundamental group and groupoid, the well-definedness of the infinitary operation $([\alpha_j])_{j\in\scrl}\mapsto \prod_{j\in\scrl}[\alpha_j]$ for all countable \textit{scattered} linear orders $\scrl$ is equivalent to the well-known \textit{homotopically Hausdorff} property (see Definition \ref{homhausdorffdef}). This paper is a study of the two homotopy invariant properties: well-definedness of transfinite $\pi_1$-products and well-definedness of transfinite $\Pi_1$-products over arbitrary countable linear orders (see Definitions \ref{tproddef} and \ref{tppdef}). As noted in \cite{BrazScattered}, it remains an open question if the well-definedness of scattered $\pi_1$-products implies the well-definedness of transfinite $\pi_1$-products.

Within the extensive literature on homotopy groups of locally non-trivial spaces, the use of infinite products is ubiquitous. This includes the algebra of the Hawaiian earring group \cite{CChegroup,deSmit,Edafreesigmaproducts,EK99subgroups,Zastrow02subgroup}, fundamental groups of one-dimensional spaces \cite{ADTW,CConedim,EK98onedimefund,EdaSpatial,FZ14menger}, and fundamental groups of planar spaces \cite{Edafreessigmaplane,FZ05}. The well-definedness of transfinite $\pi_1$-products plays an implicit role in Katsuya Eda's homotopy classification of one-dimensional Peano continua \cite{Edaonedimhomtype} and related ``automatic continuity" results for fundamental groups of one-dimensional and planar Peano continua \cite{CK,Edafreessigmaplane,Kentplanarhomom}. The well-definedness of transfinite $\Pi_1$-products as a property in its own right was first formalized in \cite[Section 7]{BFTestMap} as an intermediate property useful for proving partial converses of known implications.

Our first main result detects a fundamental difference between the well-definedness of transfinite $\pi_1$-products and transfinite $\Pi_1$-products. This is somewhat surprising since properties of groups typically pass to their groupoid counterparts and since the two are equivalent when restricted to products indexed over scattered orders.

\begin{theorem}\label{counterexamplespace}
There exists a locally path-connected metric space $X$ having well-defined transfinite $\pi_1$-products but which does not have well-defined transfinite $\Pi_1$-products.
\end{theorem}

Generalizations of covering space theory have proved useful for studying fundamental groups with non-trivial transfinite products and identifying connections with geometric and topological group theory (e.g. See \cite{BrazOpenSub,CHPgeometrylifting,FZ013caley,FZ14menger}). In \cite{FZ07}, Fischer and Zastrow set the foundation for a study of generalized universal covering maps based only on lifting properties. According to \cite{Brazcat}, this approach to generalized covering spaces is, for metric spaces, categorically, the most robust possible generalization of covering theory based on unique lifting, which gives a classification by a lattice of $\pi_1$-subgroups. Conditions sufficient for the existence of a generalized universal covering include $\pi_1$-shape injectivity \cite{FZ07} and the homotpically path-Hausdorff property \cite{FRVZ11}; we utilize the necessary and sufficient characterization for metric spaces in \cite{BFTestMap}.

The second main result of this paper proves that the existence of a generalized universal covering map is equivalent to the well-definedness of transfinite $\Pi_1$-products for any metric space.

\begin{theorem}\label{mainthm1}
A path-connected metrizable space $X$ admits a generalized universal covering in the sense of \cite{FZ07} if and only if $X$ has well-defined transfinite $\Pi_1$-products.
\end{theorem}

It is evident from Theorem \ref{mainthm1} that the space used to prove Theorem \ref{counterexamplespace} is a refinement of the example in \cite{VZ13}, i.e. a homotopically Hausdorff space without a generalized universal covering space. Together, Theorems \ref{counterexamplespace} and \ref{mainthm1} illustrate that the algebraic structure of fundamental groups alone are insufficient for characterizing the existence of generalized covering spaces; one must employ the fundamental groupoid.

We prove our results using the closure operators introduced in \cite{BFTestMap}. This approach provides convenient theoretical tools for effectively characterizing and comparing many properties of fundamental group(oid)s. While the absolute (i.e. trivial subgroup) case is of primary interest, the generality of the closure framework allows us to prove a subgroup-relative version of Theorem \ref{mainthm1} (See Theorem \ref{mainthm1general}). Such subgroup-relative results are particularly useful for distinguishing properties, e.g. in the proof of Theorem \ref{counterexamplespace}.

The remainder of this paper is structured as follows: In Section \ref{prelim}, we settle notation and recall relevant background content on paths and linear orders. In Section \ref{closuresection}, we review closure pairs and operators from \cite{BFTestMap}, focusing on the Hawaiian earring as a test space. In Section \ref{sectionrelativecw}, we consider relative CW-complexes and the preservation of closure properties after the attachment of 2-cells. In Section \ref{sectiongroupoids}, we prove Theorem \ref{counterexamplespace} by constructing a special normal subgroup and applying the results of Section \ref{sectionrelativecw}. In Section \ref{sectioncoveringspaces}, we briefly review generalized coverings spaces and use the ternary Cantor map and the dyadic arc-space from \cite{BFTestMap} to prove Theorem \ref{mainthm1general}, the equivalence of the well-definedness of $\Pi_1$-products relative to a normal subgroup $N\trianglelefteq \pionex$ and the existence of a generalized covering space corresponding to $N$. We also show these two notions are distinct for non-normal subgroups (See Example \ref{examplenonnormal}).

\section{preliminaries and notation}\label{prelim}

Most of the notation in this paper agrees with that in \cite{BFTestMap}. Throughout this paper, $X$ will denote a path-connected topological space and $x_0\in X$ will be a basepoint. The homomorphism induced on $\pi_1$ by a based map $f:(X,x)\to (Y,y)$ is denoted $f_{\#}:\pi_1(X,x)\to\pi_1(Y,y)$.

A \textit{path} is a continuous function $\alpha:[0,1]\to X$, which we call a \textit{loop} based at $x\in X$ if $\alpha(0)=\alpha(1)=x$. If $[a,b],[c,d]\subseteq \ui$ and $\alpha:[a,b]\to X$, $\beta:[c,d]\to X$ are maps, we write $\alpha\equiv\beta$ if $\alpha=\beta\circ \phi$ for some increasing homeomorphism $\phi: [a,b]\to [c,d]$; if $\phi$ is linear and if it does not create confusion, we will identify $\alpha$ and $\beta$. Under this identification, the restriction $\alpha|_{[a,b]}$ of a path $\alpha:\ui\to X$ is a path itself with a path-homotopy class $[\alpha|_{[a,b]}]$.

If $\alpha:\ui\to X$ is a path, then $\alpha^{-}(t)=\alpha(1-t)$ is the reverse path. If $\alpha_1,\alpha_2,\dots,\alpha_n$ is a sequence of paths such that $\alpha_{j}(1)=\alpha_{j+1}(0)$ for each $j$, then $\prod_{j=1}^{n}\alpha_j=\alpha_1\cdot \alpha_2\cdot\;\cdots\;\cdot \alpha_n$ is the path defined as $\alpha_j$ on $\left[\frac{j-1}{n},\frac{j}{n}\right]$. A sequence $\alpha_1,\alpha_2,\alpha_3,\dots$ of paths in $X$ is a \textit{null sequence} if $\alpha_n(1)=\alpha_{n+1}(0)$ for all $n\in\bbn$ and there is a point $x\in X$ such that every neighborhood of $x$ contains $\alpha_n([0,1])$ for all but finitely many $n$. The \textit{infinite concatenation} of such a null sequence is the path $\prod_{n=1}^{\infty}\alpha_n$ defined to be $\alpha_n$ on $\left[\frac{n-1}{n},\frac{n}{n+1}\right]$ and $\left(\prod_{n=1}^{\infty}\alpha_n\right)(1)=x$.

A path $\alpha:[a,b]\to X$ is \textit{reduced} if $\alpha$ is constant or if whenever $a\leq s<t\leq b$ with $\alpha(s)=\alpha(t)$, the loop $\alpha|_{[s,t]}$ is not null-homotopic. If $X$ is a one-dimensional metric space, then every path $\alpha:\ui\to X$ is path-homotopic within the image of $\alpha$ to a reduced path, which is unique up to reparameterization \cite{EdaSpatial}.

If $H\leq \pionex$ is a subgroup and $\alpha:\ui\to X$ is a path from $\alpha(0)=x_0$ to $\alpha(1)=x$, let $H^{\alpha}=[\alpha^{-}]H[\alpha]\leq \pi_1(X,x)$ denote the path-conjugate subgroup under basepoint change.

For basic theory of linearly ordered sets, we refer to \cite{RosensteinLO}.

\begin{definition}\label{lodef}
Let $(L,\leq)$ be a linearly ordered set.
\begin{enumerate}
\item $L$ is \textit{dense} if $L$ has more than one point and if for each $x,y\in L$ with $x<y$, there exists $z\in L$ with $x<z<y$.
\item $L$ is a \textit{scattered order} if $L$ contains no dense suborders.
\end{enumerate}
\end{definition} 
The empty order and one-point order are scattered by definition. Every countable linear order embeds as a suborder of the dense order $\bbq$.
\begin{definition}
If $A$ is a non-degenerate compact subset of $\bbr$, let $\mci(A)$ denote the set of components of $[\min(A),\max(A)]\backslash A$ equipped with the ordering inherited from $\bbr$. 
\end{definition}
Note that $\mci(A)$ is always countable and if $A$ is nowhere dense, i.e. the closure of $A$ has empty interior, then $A$ may be identified with the set of cuts of $\mci(A)$. The intuition behind the transfinite products appearing in our main results may seem somewhat masked by subgroup-relative approach of the following sections. We use the next remark to provide some intuition for the absolute version of well-definedness of transfinite products in fundamental group(oids).

\begin{remark}
Finite and infinite concatenations of paths generalize to concatenations indexed by arbitrary countable linear orders in the following way: Let $\scrl$ be a countable linear order and $A\subseteq \ui$ be a closed, nowhere dense set containing $\{0,1\}$ such that there is an order-preserving bijection $\psi:\scrl\to \mci(A)$. A system of paths $(\alpha_j)_{j\in\scrl}$ in $X$ is \textit{composable} if there exists a \textit{transfinite concatenation} path $\alpha=\prod_{j\in\scrl}\alpha_{j}$ such that $\alpha|_{\overline{\psi(j)}}=\alpha_j$ for each $j\in\scrl$. Another choice of the pair $(A,\psi)$ will result in a reparameterization of $\prod_{j\in\scrl}\alpha_{j}$.

Hence, for any countably infinite linear order $\scrl$ and composable system of paths $(\alpha_j)_{j\in\scrl}$, there is a \textit{transfinite $\Pi_1$-product} of homotopy classes $([\alpha_j])_{j\in\scrl}\mapsto \left[\prod_{j\in\scrl}\alpha_{j}\right]$; this is a \textit{transfinite $\pi_1$-product} at $x\in X$ if each $\alpha_j$ is a loop based at $x$.

This partial infinitary operation on homotopy classes (of paths or loops) is well-defined if $\prod_{j\in\scrl}\alpha_{j}$ is path-homotopic to $\prod_{j\in\scrl}\beta_{j}$ whenever $\alpha_j$ is path-homotopic to $\beta_j$ for all $j\in\scrl$. The term ``well-defined transfinite $\Pi_1$-products" refers to the well-definedness of these operations on homotopy classes for all countable linear orders $\scrl$. The term ``well-defined transfinite $\pi_1$-products" refers to the well-definedness, for all $x\in X$, of the operation restricted to loops based at $x$.
\end{remark}

\section{closure operators on subgroups of fundamental groups}\label{closuresection}

\subsection{Review of Closure Pairs and Operators}

The following definitions are from \cite{BFTestMap} where closure operators on the $\pi_1$-subgroup lattice are introduced.

\begin{definition}\label{testmap}
Suppose $(\bbt,t_0)$ is a based space, $T\leq \pionet$ is a subgroup, and $g\in \pionet$. A subgroup $H\leq \pionex$ is $(T,g)${\it-closed} if for every based map $f:(\bbt,t_0)\to (X,x_0)$ such that $f_{\#}(T)\leq H$, we also have $f_{\#}(g)\in H$. We refer to $(T,g)$ as a {\it closure pair} for the \textit{test space} $(\bbt,t_0)$.
\end{definition}
The set of $(T,g)$-closed subgroups of $\pionex$ is closed under intersection and therefore forms a complete lattice.

\begin{definition}
The $(T,g)${\it-closure} of a subgroup $H\leq \pionex$ is \[ \cl_{T,g}(H)=\bigcap\{K\leq \pionex\mid K\text{ is }(T,g)\text{-closed and }H\leq K\}.\]
\end{definition}

We refer to Section 2 of \cite{BFTestMap} for proofs of the following basic properties of these closure operators.

\begin{lemma}[Closure Operator Properties of $\cl_{T,g}$]\label{closurepropertieslemma}
Let $(T,g)$ be a closure pair. Then $\cl_{T,g}(H)=H$ if and only if $H$ is $(T,g)$-closed. Moreover,
\begin{enumerate}
\item $H\leq \cl_{T,g}(H)$,
\item $H\leq K$ implies $\cl_{T,g}(H)\leq \cl_{T,g}(K)$,
\item $\cl_{T,g}(\cl_{T,g}(H))=\cl_{T,g}(H)$,
\item if $f:(X,x_0)\to (Y,y_0)$ is a map, then $f_{\#}(\cl_{T,g}(H))\leq \cl_{T,g}(f_{\#}(H))$ in $\pi_1(Y,y_0)$.
\end{enumerate}
\end{lemma}

See \cite{BrazScattered} for a construction of $\cl_{T,g}(H)$ from $H$ using transfinite induction.
%
\begin{definition}\label{normaltestpairdef}
A closure pair $(T,g)$ for the test space $(\bbt,t_0)$ is called \textit{normal} if given any space $(X,x_0)$ and subgroup $H\leq \pionex$, $H$ is $(T,g)$-closed if and only if for every path $\alpha\in P(X,x_0)$, $H^{\alpha}$ is a $(T,g)$-closed subgroup of $\pi_1(X,\alpha(1))$.
\end{definition}

If $(T,g)$ is a normal closure pair, then the closure operator $\cl_{T,g}$ preserves the normality of subgroups. Any closure pair for a well-pointed test space $(\bbt,t_0)$ is normal.

\subsection{The Hawaiian earring as a test space}

Let $C_n\subseteq \bbr^2$ be the circle of radius $\frac{1}{n}$ centered at $\left(\frac{1}{n},0\right)$ and $\bbh=\bigcup_{n\in\bbn}C_n$ to be the usual Hawaiian earring space with basepoint $b_0=(0,0)$. Let $\ell_n(t)=\left(\frac{1}{n}(1-\cos(2\pi t)),-\frac{1}{n}\sin(2\pi t)\right)$ be the canonical counterclockwise loop traversing $C_n$. We consider the following important loops in $\bbh$, which represent a prototypical $\omega$-product, dense product, and ``densely conjugated" (see the proof of Lemma \ref{twoclosureslemma} for justification of this term) product respectively. 
\begin{enumerate}
\item $\ell_{\infty}$ denotes the infinite concatenation $\prod_{n=1}^{\infty}\ell_n$ (see Figure \ref{lingfigure}).
\begin{figure}[H]
\centering \includegraphics[height=0.6in]{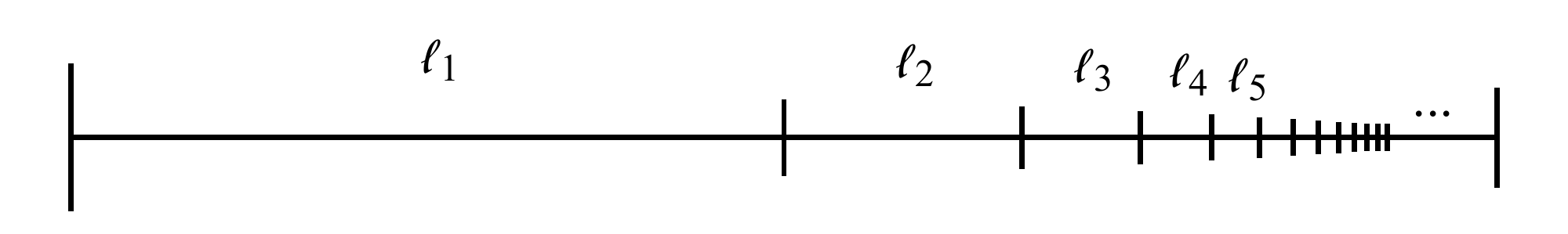}
\caption{\label{lingfigure}The loop $\ell_{\infty}$}
\end{figure}
\item Let $\mcc\subseteq \ui$ be the standard middle third Cantor set. Write $\mci(\mcc)=\{I_{n}^{k}\mid n\in\bbn,1\leq k\leq 2^{n-1}\}$ where $I_{n}^{k}$ is an open interval of length $\frac{1}{3^n}$ and, for fixed $n$, the sets $I_{n}^{k}$ are indexed by their ordering in $\mci(\mcc)$. Let $\ell_{\tau}:\ui\to \bbh$ be the transfinite concatenation $\prod_{I_{n}^{k}\in\mci(\mcc)}\ell_{2^{n-1}+k-1}$ over the dense order $\mci(\mcc)$. In other words, $\ell_{\tau}(\mcc)=b_0$ and $\ell_{\tau}|_{\overline{I_{n}^{k}}}=\ell_{2^{n-1}+k-1}$ (see Figure \ref{elltaufigure}).
\begin{figure}[H]
\centering \includegraphics[height=0.6in]{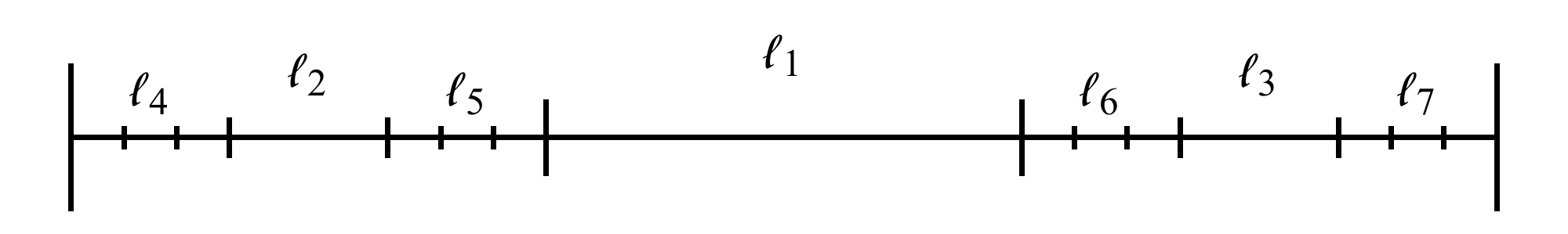}
\caption{\label{elltaufigure}The loop $\ell_{\tau}$}
\end{figure}
\item Consider the maps $f_{odd},f_{even}:\bbh\to \bbh$ satisfying $f_{odd}\circ \ell_n=\ell_{2n-1}$ and $f_{even}\circ \ell_{n}=\ell_{2n}$. Let $\ell_{odd}=f_{odd}\circ \ell_{\tau}$ and $\ell_{even}=f_{even}\circ \ell_{\tau}$. We make use of the concatenation $\ell_{odd}\cdot \ell_{even}^{-}$ (see Figure \ref{ptaufigure}).
\begin{figure}[H]
\centering \includegraphics[height=0.55in]{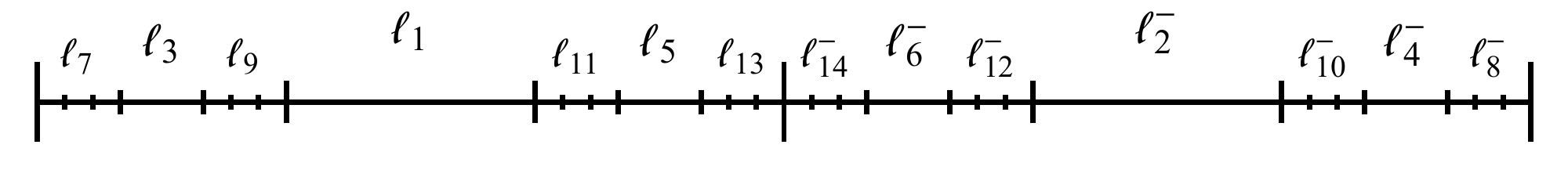}
\caption{\label{ptaufigure}The loop $\ell_{odd}\cdot \ell_{even}^{-}$}
\end{figure}
\end{enumerate}

Let $\bbhp=\bbh\cup ([-1,0]\times \{0\})$ be the Hawaiian earring with a ``whisker" attached with basepoint $\bpp=(-1,0)$ and where $\iota:\ui\to \bbhp$, $\iota(t)=(t-1,0)$ is the inclusion of the whisker (See Figure \ref{bbhpspacefig}). We define the following elements of $\pi_1(\bbhp,\bpp)$:
$$\begin{array}{lcl}
c_n=[\iota\cdot\ell_n\cdot\iota^{-}],\quad n\in\bbn; & \quad &p_n=[\iota\cdot\ell_{2n-1}\cdot\ell_{2n}^{-}\cdot\iota^{-}],\quad n\in\bbn;\\
c_{\tau}=[\iota\cdot\ell_{\tau}\cdot\iota^{-}]; & \quad  &  p_{\tau}=[\iota\cdot \ell_{odd}\cdot\ell_{even}^{-}\cdot \iota];\\
c_{\infty}=[\iota\cdot\ell_{\infty}\cdot\iota^{-}];
\end{array}$$
and the following subgroups of $\pi_1(\bbhp,\bpp)$:
$$\begin{array}{lcl}
C=\lb c_n\mid n\in\bbn\rb; & \quad &  P=\lb p_n\mid n\in\bbn\rb.
\end{array}$$

\begin{figure}[H]
\centering \includegraphics[height=1.8in]{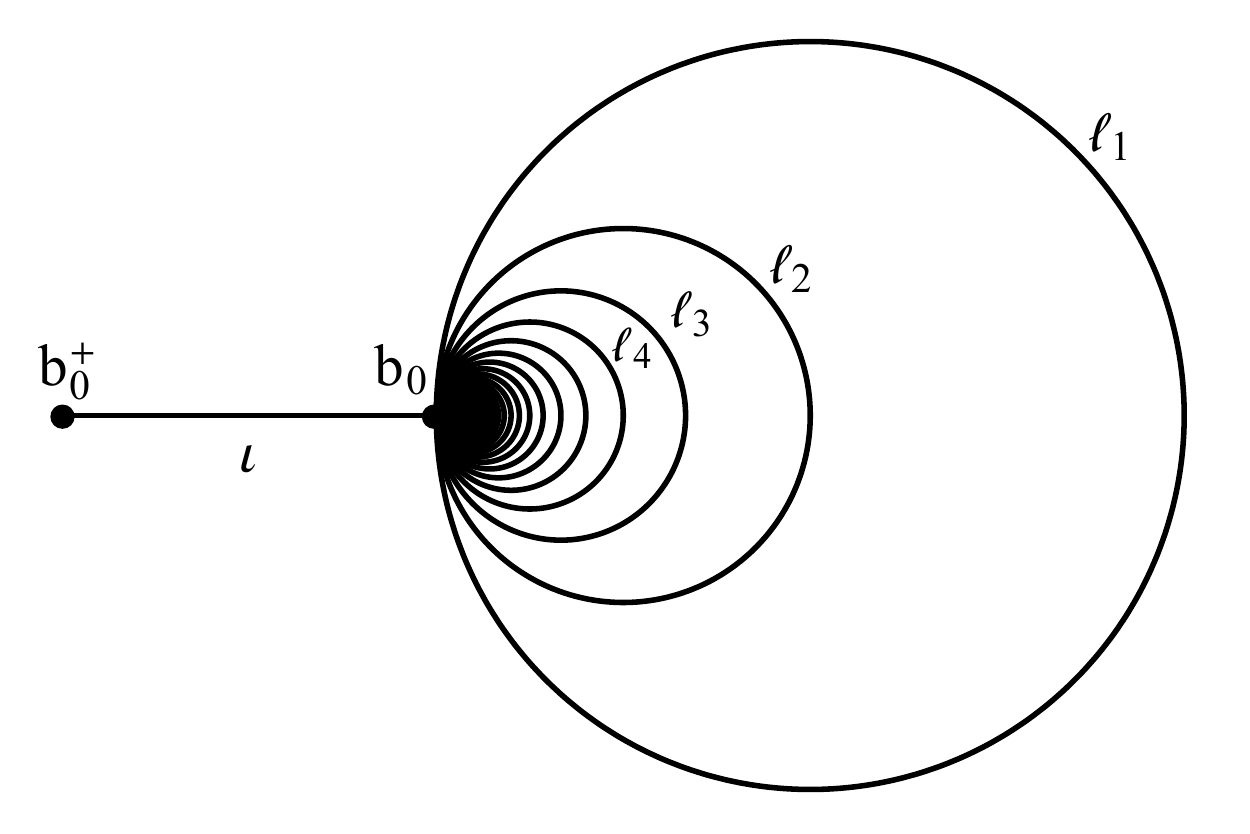}
\caption{\label{bbhpspacefig}The space $\bbhp$.}
\end{figure}

We consider the closure pairs $(C,c_{\infty})$, $(C,c_{\tau})$, and $(P,p_{\tau})$. The use of the well-pointed space $(\bbhp,\bpp)$ in place of $(\bbh,b_0)$ means that these closure pairs are normal, that is, their respective closure operators detect point-local properties at all points of a space rather than at a single basepoint.

\begin{definition}{\cite{CMRZZ08,FZ07}}\label{homhausdorffdef}
We call $X$ \textit{homotopically Hausdorff relative to a subgroup} $H\leq \pionex$ if for every $x\in X$, every path $\alpha:\ui\to X$ from $\alpha(0)=x_0$ to $\alpha(1)=x$, and every $g\in \pi_1(X,x)\backslash H^{\alpha}$, there is an open neighborhood $U$ of $x$ such that there is no loop $\delta:(\ui,\{0,1\})\to (U,x)$ with $H^{\alpha}g=H^{\alpha}[\delta]$. The space $X$ is \textit{homotopically Hausdorff} if it is homotopically Hausdorff relative to the trivial subgroup $H=1$.
\end{definition}

\begin{remark}
A space $X$ is homotopically Hausdorff if and only if for every point $x\in X$, there are no non-trivial elements of $\pi_1(X,x)$ that have a representative loop in every neighborhood of $x$.
\end{remark}

The two propositions to follow are combinations of results in Section 3 of \cite{BFTestMap}.

\begin{proposition}\label{homhauscharprop}
For any space $X$ and subgroup $H\leq \pionex$, consider the following three properties:
\begin{enumerate}
\item $X$ is homotopically Hausdorff relative to $H$,
\item $H$ is $(C,c_{\tau})$-closed,
\item $H$ is $(C,c_{\infty})$-closed.
\end{enumerate}
In general, (1) $\Rightarrow$ (3) and (2) $\Rightarrow$ (3). If $X$ is first countable, then (3) $\Rightarrow$ (1). If $H$ is normal, then (3) $\Rightarrow$ (2).
\end{proposition}

\begin{example}\label{fandscex}
A topological space $S$ is a \textit{scattered space} if every non-empty subspace of $S$ contains an isolated point. If $X$ is a one-dimensional metric space and $A\subseteq X$ is closed, we define: 
\begin{enumerate}
\item[] $F(X,A)=\{[\alpha]\in \pionex\mid \alpha^{-1}(A)\text{ is finite or }\alpha\text{ is constant}\}$,
\item[] $\sct(X,A)=\{[\alpha]\in \pionex\mid \alpha^{-1}(A)\text{ is a scattered space or }\alpha\text{ is constant}\}$.
\end{enumerate}
It is shown in \cite{BrazScattered} that $\cl_{C,c_{\infty}}(F(X,A))=Sc(X,A)$. For instance, the subgroup $Sc=Sc(\bbh,\{b_0\}) \leq \pi_1(\bbh,b_0)$ of scattered words \cite{EK99subgroups} is an example of a non-normal subgroup which is $(C,c_{\infty})$-closed but not $(C,c_{\tau})$-closed.
\end{example}

\begin{definition}\label{tproddef}
We say a space $X$ has {\it well-defined transfinite $\pi_1$-products relative to a subgroup} $H\leq \pionex$ provided that for every pair of maps $a,b:(\bbhp,\bpp)\to (X,x_0)$ such that $a\circ\iota=b\circ\iota$ and $Ha_{\#}(c_n)=Hb_{\#}(c_n)$ for all $n\in\bbn$, we have $Ha_{\#}(g)=Hb_{\#}(g)$ for all $g\in\pi_1(\bbhp,\bpp)$. We say $X$ has {\it well-defined transfinite $\pi_1$-products} if $X$ has well-defined transfinite $\pi_1$-products relative to $H=1$.
\end{definition}

\begin{remark}
A space $X$ has well-defined transfinite $\pi_1$-products if and only if for every pair of based maps $f,g:(\bbh,b_0)\to(X,x)$ such that $f\circ\ell_n\simeq g\circ\ell_n$ for all $n\in\bbn$, we have $f_{\#}=g_{\#}$.
\end{remark}

\begin{proposition}\label{transfiniteproductprop}
For any space $X$ and subgroup $H\leq \pionex$, consider the following four properties:
\begin{enumerate}
\item $X$ has well-defined transfinite $\pi_1$-products relative to $H$,
\item $H$ is $(P,p_{\tau})$-closed,
\item $H$ is $(C,c_{\tau})$-closed,
\item $H$ is $(C,c_{\infty})$-closed.
\end{enumerate}
In general, (1) $\Leftrightarrow$ (2) $\Rightarrow$ (3) $\Rightarrow$ (4). If $H$ contains the commutator subgroup of $\pi_1(X,x_0)$, then all four are equivalent.
\end{proposition}

\begin{example}
If $f:\bbhp\to\bbhp$ is the map defined so that $f\circ\iota=\iota$ and $f\circ\ell_n=\ell_{2n-1}\cdot\ell_{2n}^{-}$, $n\in\bbn$, then $\cl_{C,c_{\tau}}(P)=f_{\#}(\pi_1(\bbhp,\bpp))$ is an example of a non-normal subgroup of $\pi_1(\bbhp,\bpp)$, which is $(C,c_{\tau})$-closed but not $(P,p_{\tau})$-closed (see the proof of \cite[Theorem 3.25]{BFTestMap}). There is no known \textit{normal} counterexample.
\end{example}

\section{Relative CW-Complexes}\label{sectionrelativecw}

To simplify the proof of Theorem \ref{counterexamplespace}, we consider \textit{relative CW-complexes}, i.e. spaces constructed by attaching 2-cells to spaces that are not necessarily CW-complexes.

\begin{proposition}\label{cweasy}
Suppose $(T,g)$ is a closure pair for test space $(\bbt,t_0)$, $(X,x_0)$ is a space, and $Y=X\cup\bigcup_{\beta\in S} D^{2}_{\beta}$ is the space obtained from $X$ by attaching 2-cells along a family $S$ of loops in $X$ based at $x_0$. If $j:X\to Y$ is the inclusion map and the trivial subgroup of $\pi_1(Y,x_0)$ is $(T,g)$-closed, then $\ker j_{\#}$ is $(T,g)$-closed.
\end{proposition}

\begin{proof}
Suppose $1\leq \pi_1(Y,x_0)$ is $(T,g)$-closed and $f:(\bbt,t_0)\to (X,x_0)$ is a map such that $f_{\#}(T)\leq \ker j_{\#}$. Then $j\circ f:(\bbt,t_0)\to (Y,x_0)$ satisfies $(j\circ f)_{\#}(T)=1$ so by assumption, we have $(j\circ f)_{\#}(g)=1$. Thus $f_{\#}(g)\in\ker j_{\#}$.
\end{proof}

The next proposition shows the converse of Proposition \ref{cweasy} holds in a special case. For each $n\in\bbn$, both $\bbh_{\leq n}^{+}=[-1,0]\times \{0\}\cup \bigcup_{k=1}^{n}C_k$ and $\bbh_{\geq n}^{+}=[-1,0]\times \{0\}\cup \bigcup_{k=n}^{\infty}C_k$ are retracts of $\bbhp$. Let $r_n:\bbhp\to \bbh_{\leq n}^{+}$ denote the retraction collapsing $\bigcup_{k=n+1}^{\infty}C_k$ to $b_0$ and identify the free subgroup $C_{\leq n}=\lb c_k\mid 1\leq k\leq n\rb$ of $\pi_1(\bbhp,\bpp)$ with the fundamental group $\pi_1(\bbh_{\leq n}^{+},\bpp)$. Since $\bbhp$ is one-dimensional, the retractions $(r_n)_{\#}$ induce a natural injection $\pi_1(\bbhp,\bpp)\to \check{\pi}_{1}(\bbhp,\bpp)=\varprojlim_{n} C_{\leq n}$ to the first shape homotopy group. Consequently, $\pi_1(\bbhp,\bpp)$ splits as the free product $C_{\leq n}\ast \pi_1(\bbh_{\geq n+1}^{+},\bpp)$ for each $n\in\bbn$. 

\begin{lemma}\label{cellsandclosurelemma}
Suppose $(T,g)$ is a closure pair for $(\bbhp,\bpp)$ such that for all $n\in\bbn$, there exists an $m\geq n$ such that $(r_{m})_{\#}(g)$ lies in the normal closure of $(r_{m})_{\#}(T)$ in $C_{\leq m}$. Suppose $(X,x_0)$ is a space and $Y=X\cup\bigcup_{\beta\in S} D^{2}_{\beta}$ is the space obtained from $X$ by attaching 2-cells along a family $S$ of loops in $X$ based at $x_0$. If $j:X\to Y$ is the inclusion map, then $\ker j_{\#}$ is $(T,g)$-closed if and only if the trivial subgroup of $\pi_1(Y,x_0)$ is $(T,g)$-closed.
\end{lemma}

\begin{proof}
Suppose $(T,g)$ satisfies the hypothesis in the statement of the lemma. The first direction is Proposition \ref{cweasy}. Suppose $\ker j_{\#}$ is $(T,g)$-closed and consider a map $f:(\bbhp,\bpp)\to (Y,x_0)$ such that $f_{\#}(T)=1$. It suffices to show $f_{\#}(g)=1$. Set $\alpha=f\circ\iota$, $\gamma_n=f\circ \ell_n$, $n\in\bbn$, and let $K_m$ denote the normal closure of $(r_{m})_{\#}(T)$ in $C_{\leq m}$.

If $f(b_0)$ lies in the interior of a 2-cell, then there exists an $m\in\bbn$ such that $f_{\#}$ vanishes on $\pi_1(\bbh_{\geq m+1}^{+},\bpp)$. By our assumption on $(T,g)$, we may choose $m$ large enough so that $(r_{m})_{\#}(g)\in K_m$. Since $\pi_1(\bbhp,\bpp)$ may be identified with the free product $C_{\leq m}\ast \pi_1(\bbh_{\geq m+1}^{+},\bpp)$, it follows that $f_{\#}\circ (r_{m})_{\#}(h)=f_{\#}(h)$ for all $h\in \pi_1(\bbhp,\bpp)$. Since $f_{\#}((r_{m})_{\#}(T))=f_{\#}(T)=1$, we have $(r_{m})_{\#}(T)\leq \ker f_{\#}$ and thus $K_m\leq \ker f_{\#}$ by the normality of $\ker f_{\#}$. Therefore, 
\[f_{\#}(g)=f_{\#}((r_{m})_{\#}(g))\in f_{\#}(K_m)=1.\]
On the other hand, suppose $f(b_0)\in X$. Since $j_{\#}$ is onto, we may find a path $\alpha ':\ui\to X$ from $x_0$ to $f(b_0)$, which is path-homotopic to $\alpha$ in $Y$. Let $U_{\beta}$ be an open set in the 2-cell $D^{2}_{\beta}$, which deformation retracts on the boundary $\partial D^{2}_{\beta}$. Let $Z=X\cup\bigcup_{\beta}U_{\beta}$ and note that there is a deformation retraction $\phi:Z\to X$. Since $Z$ is an open neighborhood of $f(b_0)$, there is an $m\in\bbn$ such that $f(\bbh_{\geq m+1})\subset Z$. For each $n=1,2,\dots,m$, there is a loop $\gamma_{n}':S^1\to X$ based at $x_0$ which is homotopic to $\gamma_n$ in $Y$. Define $f':\bbhp\to X$ by $f'\circ \iota=\alpha '$, $f'|_{\bbh_{\geq m+1}}=\phi\circ f|_{\bbh_{\geq m+1}}$, and $f'\circ \ell_n=\gamma_{n}'$ for $n=1,2,\dots,m$. Notice that $f'$ has image in $X$ and $f$ is homotopic to $j\circ f'$ in $Y$. Since $f_{\#}=(j\circ f')_{\#}$, we have $j_{\#}((f')_{\#}(T))=f_{\#}(T)=1$ and thus $(f')_{\#}(T)\leq \ker j_{\#}$. We have assumed that $\ker j_{\#}$ is $(T,g)$-closed and conclude that $(f')_{\#}(g)\in \ker j_{\#}$. Thus $f_{\#}(g)=j_{\#}((f')_{\#}(g))=1$.
\end{proof}

\begin{lemma}\label{twoclosureslemma}
The closure pairs $(C,c_{\infty})$, $(C,c_{\tau})$, and $(P,p_{\tau})$ satisfy the hypothesis of Lemma \ref{cellsandclosurelemma}.
\end{lemma}

\begin{proof}
Suppose $n\in\bbn$. For $(C,c_{\infty})$ and $(C,c_{\tau})$, we may take $m=n$ since $(r_{m})_{\#}(C)=C_{\leq m}$.

For $(P,p_{\tau})$, let $K_m$ be the normal closure of $(r_{m})_{\#}(P)$ in $C_{\leq m}$. To verify the hypothesis of Lemma \ref{cellsandclosurelemma}, will show that $(r_{2n})_{\#}(p_{\tau})\in K_{2n}$ for all $n\in\bbn$. Note that since $\bbhp_{\leq 2n-2}$ is a retract of $\bbhp_{\leq 2n}$, we may identify $K_{2n-2}$ canonically as a non-normal subgroup of $K_{2n}$.

A \textit{basic factorization} of $(r_{2n})_{\#}(p_{\tau})$ is a product $w_{odd}v_{odd}v_{even}^{-1}w_{even}^{-1}$ in $C_{\leq 2n}$ where 
\begin{enumerate}
\item $w_{odd}v_{odd}v_{even}^{-1}w_{even}^{-1}$ is already a reduced representative of $(r_{2n})_{\#}(p_{\tau})$.
\item $w_{odd},v_{odd}\in\lb c_1,c_3,\dots c_{2n-1}\rb$ and $w_{even},v_{even}\in\lb c_2,c_4,\dots c_{2n}\rb$,
\item $w_{odd}$ and $w_{even}$ have equal word length (which may be $0$),
\item $v_{odd}$ and $v_{even}$ have equal word length (which may be $0$),
 \end{enumerate}
As a necessary convention, we consider the case where $w_{odd}=w_{even}=1$ are empty and $v_{odd}v_{even}^{-1}=(r_{2n})_{\#}(p_{\tau})$ and the case where $v_{odd}=v_{even}=1$ are empty and $w_{odd}w_{even}^{-1}=(r_{2n})_{\#}(p_{\tau})$ to be distinct basic factorizations.

Recalling the structure of $p_{\tau}$, note that $(r_{2n})_{\#}(p_{\tau})$ has $n+1$ basic factorizations. We prove the following by induction on $n$: If $w_{odd}v_{odd}v_{even}^{-1}w_{even}^{-1}$ is a basic factorization of $(r_{2n})_{\#}(p_{\tau})$, then $v_{odd}v_{even}^{-1}$ and $w_{odd}w_{even}^{-1}$ are elements of $K_{2n}$.

For $n=1$, we have $(r_{2})_{\#}(p_{\tau})=[\ell_1\cdot\ell_{2}^{-}]=p_1\in (r_{2})_{\#}(P)\leq K_2$. Suppose the hypothesis holds for $n-1$. By the definition of $p_{\tau}$, there is a unique factorization \[(r_{2n})_{\#}(p_{\tau})=w_{odd}c_{2n-1}v_{odd}v_{even}^{-1}c_{2n}^{-1}w_{even}^{-1}\]
where $w_{odd}v_{odd}v_{even}^{-1}w_{even}^{-1}$ is a basic factorization of $(r_{2n-2})_{\#}(p_{\tau})$. By our induction hypothesis, we have $w_{odd}w_{even}^{-1},v_{odd}v_{even}^{-1}\in K_{2n-2}$. Since $K_{2n-2}\leq K_{2n}$, we have $w_{odd}w_{even}^{-1},v_{odd}v_{even}^{-1}\in K_{2n}$. Therefore, the equality{\footnotesize{
\[
w_{odd}c_{2n-1}v_{odd}v_{even}^{-1}c_{2n}^{-1}w_{even}^{-1} =(w_{odd}((c_{2n-1}c_{2n}^{-1})(c_{2n}(v_{odd}v_{even}^{-1})c_{2n}^{-1}))w_{odd}^{-1})(w_{odd}w_{even}^{-1})\]}}
shows that $(r_{2n})_{\#}(p_{\tau})$ an element of $K_{2n}$. This completes the induction.

Now, for any $n\in\bbn$, and using any basic factorization $ w_{odd}v_{odd}v_{even}^{-1}w_{even}^{-1}$ of $(r_{2n})_{\#}(p_{\tau})$, we see that
\[
(r_{2n})_{\#}(p_{\tau}) = w_{odd}v_{odd}v_{even}^{-1}w_{even}^{-1}= w_{odd}w_{even}^{-1}(w_{even}(v_{odd}v_{even}^{-1})w_{even}^{-1})
\]
is an element of $K_{2n}$.
\end{proof}

\section{Well-definedness of dense products in fundamental groupoids}\label{sectiongroupoids}

\begin{definition}\label{tppdef}
A space $X$ has {\it well-defined transfinite $\Pi_1$-products relative to }$H\leq \pionex$ if for every closed set $A\subseteq \ui$ containing $\{0,1\}$ and paths $\alpha,\beta:(\ui,0)\to (X,x_0)$ such that $\alpha|_A=\beta|_A$ and $[\alpha|_{[0,b]}\cdot\beta|_{[a,b]}^{-}\cdot \alpha|_{[0,a]}^{-}]\in H$ for every component $(a,b)$ of $\ui\backslash A$, we have $[\alpha\cdot \beta^{-}]\in H$. A space $X$ has {\it well-defined transfinite $\Pi_1$-products} if $X$ has well-defined transfinite $\Pi_1$-products relative to the trivial subgroup $H=1$.
\end{definition}

As noted in \cite[Remark 7.2]{BFTestMap}, in the previous definition, one only needs to consider closed nowhere dense sets $A$.

Recall that $\mcc$ is the Cantor set and $\mci(\mcc)$ is a countable dense order. For $I=(a,b)\in\mci(\mcc)$, let $C_I=\left\{(x,y)\in\bbr^2\mid y\geq 0,\left(x-\frac{a+b}{2}\right)^2+y^2=\left(\frac{b-a}{2}\right)^2\right\}$ be the semicircle whose boundary is $\{(a,0),(b,0)\}$. Let $\bbw=B\cup \bigcup_{I\in\mci(\mcc)}C_I$ where $B=[0,1]\times\{0\}$ is the \textit{base-arc} and $w_0=(0,0)$ is the basepoint (see Figure \ref{waspace}).

\begin{figure}[H]
\centering \includegraphics[height=.9in]{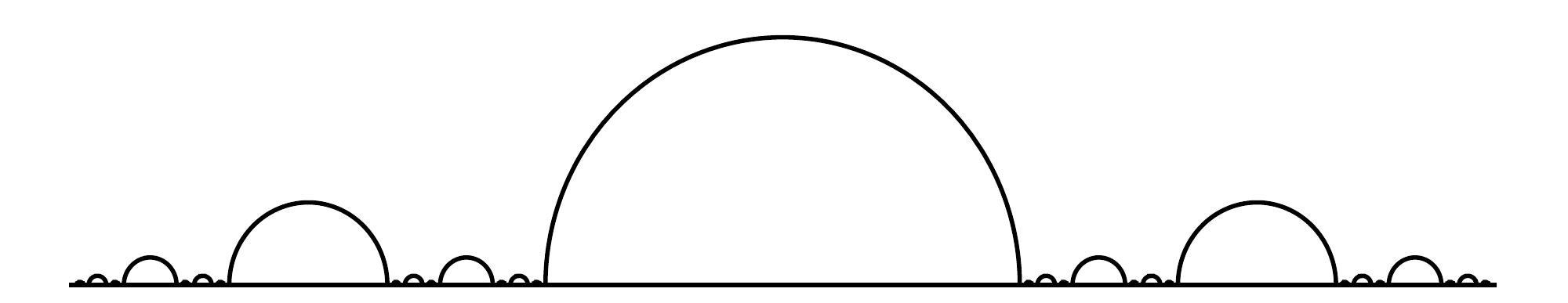}
\caption{\label{waspace}The space $\bbw$.}
\end{figure}

For $I=(a,b)\in\mci(\mcc)$, let $\lambda_I:\ui\to B$ be the path $\lambda_I(t)=(bt+a(1-t),0)$ and $\upsilon_{I}:\ui\to C_I$ be the path so that if $r:\bbw\to B$ is the projection onto the x-axis, then $r\circ\upsilon_{I}=\lambda_I$. Let $\lambda(t)=(t,0)$ and $\upsilon:\ui\to (\mcc\times \{0\})\cup\bigcup_{I\in\mci(\mcc)}C_I$ be the path such that $r\circ\upsilon=\lambda$, i.e. the respective transfinite concatenations $\upsilon=\prod_{I\in\mci(\mcc)}\upsilon_I$ and $\lambda=\prod_{I\in\mci(\mcc)}\lambda_I$.

Let $W$ be the subgroup of $\pi_1(\bbw,w_0)$ generated by the elements $w_I=[\upsilon_{[0,b]}\cdot \lambda_{[a,b]}^{-}\cdot \upsilon_{[0,a]}^{-}]$, $I=(a,b)\in\mci(\mcc)$ and let $w_{\infty}=[\upsilon\cdot\lambda^{-}]$. Although $\bbw$ is not well-pointed at $w_0$, the self-similarity of $\bbw$ ensures that $(W,w_{\infty})$ is a normal closure pair for $(\bbw,w_0)$ (see \cite[Proposition 7.4]{BFTestMap}).

\begin{proposition}\label{transfinitepathprodchar}\cite[Proposition 7.5]{BFTestMap}
$X$ has well-defined transfinite $\Pi_1$-products relative to $H\leq \pionex$ if and only if $H$ is $(\bbw,w_{\infty})$-closed.
\end{proposition}

We now construct a normal subgroup which is $(P,p_{\tau})$-closed but not $(W,w_{\infty})$-closed.

If $\gamma:\ui\to X$ is a loop in a space, let $\langle \gamma\rangle\in H_1(X)$ denote the class of $\gamma$ in first singular homology. For each $I=(a,b)\in\mci(\mcc)$, let $K_I=C_I\cup ([a,b]\times \{0\})\subset \bbw$ and note that we may identify $H_1(K_I)$ with $\bbz$ by identifying $\langle\upsilon|_{I}\cdot \lambda|_{I}^{-}\rangle=1$. There is a natural retraction $q_I:\bbw\to K_I$ such that $q_I(x,y)=(a,0)$ for $x\leq a$ and $q_I(x,y)=(b,0)$ if $x\geq b$. In combination with the Hurewicz homomorphisms, these maps induce a homomorphism $\phi:\pi_1(\bbw,w_0)\to \prod_{I\in \mci(\mcc)}H_1(K_I)$ given by $\phi([\alpha])=(\langle q_I\circ \alpha\rangle)_{I\in\mci(\mcc)}$.

The \textit{support} of an element $g=(g_I)_{I\in\mci(\mcc)}\in \prod_{I\in \mci(\mcc)}H_1(K_I)$ is the set \[\supp(g)=\{I\in\mci(\mcc)\mid g_I\neq 0\}.\] Note that if loops $\alpha$ and $\beta$ are freely homotopic loops in $\bbw$, then $\supp((\langle q_I\circ \alpha\rangle)_{I\in\mci(\mcc)})=\supp((\langle q_I\circ \beta\rangle)_{I\in\mci(\mcc)})$. We are interested in the set \[N_0=\{[\alpha]\in\pi_1(\bbw,w_0)\mid \supp(\phi([\alpha]))\text{ is a scattered suborder of }\mci(\mcc)\},\]which contains the homotopy classes of loops that have non-zero winding number around a scattered ordering of the simple closed curves $K_I$, $I\in\mci(\mcc)$. 

\begin{proposition}\label{basicsubgroupprop}
The set \[N_0=\{[\alpha]\in\pi_1(\bbw,w_0)\mid \supp(\phi([\alpha]))\text{ is a scattered suborder of }\mci(\mcc)\}\] is a normal subgroup of $\pi_1(\bbw,w_0)$ such that 
\begin{enumerate}
\item $[\pi_1(\bbw,w_0),\pi_1(\bbw,w_0)]\leq N_0$,
\item $W\leq N_0$,
\item and $w_{\infty}\notin N_0$.
\end{enumerate}
\end{proposition}

\begin{proof}
Suppose $g,h\in N_0$ where $\phi(g)=(g_I)_{I\in\mci(\mcc)}$ and $\phi(h)=(h_I)_{I\in\mci(\mcc)}$. Then $\supp(\phi(g))$ and $\supp(\phi(h))$ are scattered suborders of $\mci(\mcc)$. Note that $\phi(gh^{-1})=(g_I-h_I)_{I\in\mci(\mcc)}$. The only way $g_I-h_I$ is nonzero is if one of $g_I,h_I$ is nonzero. Hence $\supp(\phi(gh^{-1}))\subseteq \supp(\phi(g))\cup \supp(\phi(h))$. Since the set of scattered suborders of any linear order is closed under both suborders and finite unions, $\supp(gh^{-1})$ is scattered. Therefore, $gh^{-1}\in N_0$, proving $N_0$ is a subgroup. If $g\in N_0$ and $h\in \pi_1(\bbw,w_0)$, then $\phi(hgh^{-1})=\phi(g)$ since $\prod_{I\in\mci(\mcc)}H_1(K_I)$ is abelian. Thus $\supp(\phi(hgh^{-1}))=\supp(\phi(g))$ is scattered, giving $hgh^{-1}\in N_0$. We conclude that $N_0$ is a normal subgroup.

Since the product $\prod_{I\in\mci(\mcc)}H_1(K_I)$ is abelian, $\ker\phi$ contains the commutator subgroup $[\pi_1(\bbw,w_0),\pi_1(\bbw,w_0)]$. Since the empty order is scattered, we have $\ker\phi=\{g\in \pi_1(\bbw,w_0)\mid \supp(\phi(g))=\emptyset\}\leq N_0$. Thus $[\pi_1(\bbw,w_0),\pi_1(\bbw,w_0)]\leq N_0$.

If $J\in \mci(\mcc)$ and $\phi(w_{J})=(g_I)_{I\in\mci(\mcc)}$, then $g_I=1$ if $I=J$ and $g_I=0$ otherwise. Since $\supp(\phi(w_I))$ contains a single element, it is a scattered order. Therefore, $w_J\in N_0$ for all $J\in\mci(\mcc)$, giving $W\leq N_0$. 

Finally, if $\phi(w_{\infty})=(g_I)_{I\in\mci(\mcc)}$, then $g_I=1$ for all $I\in\mci(\mcc)$. Since $\supp(\phi(w_{\infty}))=\mci(\mcc)$ is a dense order, we have $w_{\infty}\notin N_0$.
\end{proof}
\begin{proposition}\label{cinfclosed}
The subgroup $N_0\trianglelefteq \pi_1(\bbw,w_0)$ is $(C,c_{\infty})$-closed.
\end{proposition}

\begin{proof}
Suppose $f:(\bbhp,\bpp)\to (\bbw,w_0)$ is a based map such that $f_{\#}(C)\leq N_0$. Since $\bbw$ is one-dimensional, we may assume the path $\alpha=f\circ\iota$ and loops $\gamma_n=f\circ\ell_n$, $n\in\bbn$ are reduced. Set $\gamma=\prod_{n=1}^{\infty}\gamma_n$. We seek to show $f_{\#}(c_{\infty})=[\alpha\cdot\gamma\cdot\alpha^{-}]\in N_0$. Since $\bbw$ is locally contractible at all points except those in $\mcc\times \{0\}$, we may assume $f(b_0)=(k,0)\in \mcc\times \{0\}$. Moreover, since $\gamma$ and $\alpha\cdot\gamma\cdot\alpha^{-}$ are homologous, it is enough to show $\supp(\phi(f_{\#}(c_{\infty})))=\supp\left((\lb q_I\circ \gamma\rb)_{I\in\mci(\mcc)}\right)$ is a scattered order.

Suppose, to the contrary, that $\supp\left((\lb q_I\circ \gamma\rb)_{I\in\mci(\mcc)}\right)$ contains a dense suborder $\scrl\subseteq \mci(\mcc)$. Pick any $(a_0,b_0)\in \scrl$. If $b_0\leq k$, let $\scrl '=\{(a,b)\in\scrl\mid (a,b)<(a_0,b_0)\}$ and if $k\leq a_0$, let $\scrl '=\{(a,b)\in\scrl\mid (a_0,b_0)< (a,b)\}$. Either way, $\scrl '$ is a dense suborder of $\supp\left((\lb q_I\circ \gamma\rb)_{I\in\mci(\mcc)}\right)$. Set $\epsilon =b_0-a_0$ and notice that if $|x-k|<\epsilon$, then $x\notin \overline{\cup\{I\mid I\in\scrl '\}}$. Find integer $N>1$ such that if $n\geq N$, then the image of $\gamma_n$ lies in $U=((k-\epsilon,k+\epsilon)\times [0,1/6])\cap\bbw$. Since $\zeta=\prod_{n=N}^{\infty}\gamma_n$ has image in $U$, we have $\langle q_I\circ \zeta \rangle=0$ for all $I\in\scrl '$.

Fix $J\in\scrl '$ and recall that $\langle q_J\circ \gamma\rangle \neq 0$. Since $[\gamma]=[\gamma_1\cdot \gamma_2\cdots \gamma_{N-1}][\zeta]$, we have the following in $H_1(K_J)$:
\[
0\neq \langle q_J\circ \gamma\rangle = \langle q_J\circ (\gamma_1\cdot \gamma_2\cdots \gamma_{N-1})\rangle +\langle q_J\circ \zeta \rangle= \langle q_J\circ (\gamma_1\cdot \gamma_2\cdots \gamma_{N-1})\rangle\]
Therefore, $\scrl '$ is a dense suborder contained in \[\supp(\phi([\gamma_1\cdot \gamma_2\cdots \gamma_{N-1}]))=\supp(\phi(f_{\#}(c_1c_2\dots c_{N-1}))).\] Thus $c_1c_2\cdots c_{N-1}\in C$ but $f_{\#}(c_1c_2\dots c_{N-1})\notin N_0$; a contradiction of $f_{\#}(C)\leq N_0$.
\end{proof}

\begin{theorem}\label{counterexamplethm}
The normal subgroup $N_0$ of $\pi_1(\bbw,w_0)$ is $(P,p_{\tau})$-closed but not $(W,w_{\infty})$-closed.
\end{theorem}

\begin{proof}
In light of (2) and (3) of Proposition \ref{basicsubgroupprop}, the identity map of $\bbw$ suffices to show that $N_0$ is not $(W,w_{\infty})$-closed. According to Proposition \ref{cinfclosed} and (1) of Proposition \ref{basicsubgroupprop}, $N_0$ is $(C,c_{\infty})$-closed and contains the commutator subgroup of $\pi_1(\bbw,w_0)$. Therefore, by Proposition \ref{transfiniteproductprop}, $N_0$ is $(P,p_{\tau})$-closed.
\end{proof}

\begin{proof}[Proof of Theorem \ref{counterexamplespace}]
Consider the normal subgroup $N_0\leq \pi_1(\bbw,w_0)$ from Theorem \ref{counterexamplethm}, which is $(P,p_{\tau})$-closed but not $(W,w_{\infty})$-closed. Let $X$ be the space obtained by attaching 2-cells to $\bbw$ along a set of loops representing generators of $N_0$. According to Proposition \ref{cweasy}, the trivial subgroup of $\pi_1(X,w_0)$ is not $(W,w_{\infty})$-closed. However, by Lemmas \ref{cellsandclosurelemma} and \ref{twoclosureslemma}, the trivial subgroup of $\pi_1(X,w_0)$ is $(P,p_{\tau})$-closed. Therefore, $X$ has well-defined transfinite $\pi_1$-products but does not have well-defined transfinite $\Pi_1$-products. Although $\bbw$ is a Peano continuum, with the relative CW-topology, $X$ is not metrizable since it is not first countable. However, using a standard approach (e.g. see \cite{VZ13}), one can endow $X$ with a coarser metrizable topology without changing homotopy type to obtain a locally path-connected metric space.
\end{proof}

\begin{corollary}
If $\lb\lb W\rb\rb$ is the normal closure of $W$ in $\pi_1(\bbw,w_0)$, then $w_{\infty}\notin \cl_{P,p_{\tau}}(\lb\lb W\rb\rb)$.
\end{corollary}

\section{Transfinite $\Pi_1$-products and generalized covering spaces}\label{sectioncoveringspaces}

\subsection{A brief review of generalized covering space theory}

The following definition appears in \cite{Brazcat} under the name ``$\mathbf{lpc_{0}}$-covering" and agrees with the notion of generalized regular covering in \cite{FZ07} for normal subgroups.

\begin{definition} \label{gencovdef}
A map $p:\hX\to X$ is a \textit{generalized covering map} if
\begin{enumerate}
\item $\hX$ is nonempty, path connected, and locally path connected,
\item for every path-connected, locally path-connected space $Y$, point $\hx\in\hX$, and based map $f:(Y,y)\to (X,p(\hx))$ such that $f_{\#}(\pi_1(Y,y))\leq p_{\#}(\pi_1(\hX,\hx))$, there is a unique map $\widehat{f}:(Y,y)\to (\hX,\hx)$ such that $p\circ \widehat{f}=f$.
\end{enumerate}
If $\hX$ is simply connected, we call $p$ a \textit{generalized universal covering map}.
\end{definition}

By definition, a generalized covering map $p:\hX\to X$ has the unique path-lifting property: if $\alpha,\beta:\ui\to X$ are paths with $\alpha(0)=\beta(0)$ and $p\circ \alpha=p\circ \beta$, then $\alpha=\beta$. Moreover, if $p(\hx)=x_0$, $p$ is characterized up to equivalence by the conjugacy class of the subgroup $H=p_{\#}(\pi_1(\hX,\hx))\leq \pionex$. Additionally, when such a generalized covering space exists, it is equivalent to a construction from classical covering space theory (See \cite{Spanier66}): given a subgroup $H\leq \pionex$, let $\tXh=\pxxo/\mathord{\sim}$ where $\alpha\sim \beta$ if and only if $\alpha(1)=\beta(1)$ and $[\alpha\cdot\beta^{-}]\in H$. The equivalence class of $\alpha$ is denoted $H[\alpha]$ and $\txh$ denotes the equivalence class of the constant path at $x_0$. We give $\tXh$ the topology generated by the sets $B(H[\alpha],U)=\left\{H[\alpha\cdot\epsilon]\mid\epsilon([0,1])\subseteq U\right\}$ where $U$ is an open neighborhood of $\alpha(1)$ in $X$. Let $p_H:\tXh\to X$ denote the endpoint projection map defined as $p_H(H[\alpha])=\alpha(1)$. 

\begin{lemma}\label{gencovtopologylemma}\cite[Theorem 5.11]{Brazcat}
For any subgroup $H\leq \pionex$, the following are equivalent:
\begin{enumerate}
\item $p_H$ has the unique path lifting property,
\item $p_H$ is a generalized covering map,
\item $(p_H)_{\#}(\pi_1(\tXh,\txh))=H$,
\item $X$ admits a generalized covering $p:(\hX,\hx)\to (X,x_0)$ such that $p_{\#}(\pi_1(\hX,\hx))$ $=H$.
\end{enumerate}
\end{lemma}

Consequently, the existence of generalized covering maps depends entirely on whether or not the constructed map $p_H$ has the unique path-lifting property. 

\subsection{The dyadic arc space}

A pair of integers $(n,j)$ is \textit{dyadic unital} if the dyadic rational $\frac{2j-1}{2^n}$ lies in $(0,1)$. For each dyadic unital pair, let $\bbd(n,j)$ denote the upper semicircle $\left\{(x,y)\in \bbr^2\Big|\left(x-\frac{2j-1}{2^n}\right)^2+y^2=\left(\frac{1}{2^n}\right)^2\text{, }x\geq 0\right\}$. We consider the union $\bbd=B\cup \bigcup_{(n,j)} \bbd(n,j)$ over all dyadic unital pairs as a subspace of $\bbr^2$ and with basepoint $d_0=(0,0)$. Let $\ell_{n,j}:[0,1]\to \bbd(n,j)$ be the arc $\ell_{n,j}(t)=\left(\frac{t+j-1}{2^{n-1}}, \frac{1}{2^{n-1}}\sqrt{t-t^2}\right)$ from $\left(\frac{j-1}{2^{n-1}},0\right)$ to $\left(\frac{j}{2^{n-1}},0\right)$ and $\lambda_{\infty}(t)=(t,0)$ be the unit speed path on the base-arc (See Figure \ref{thespacedfig}).

\begin{figure}[H]
\centering \includegraphics[height=2in]{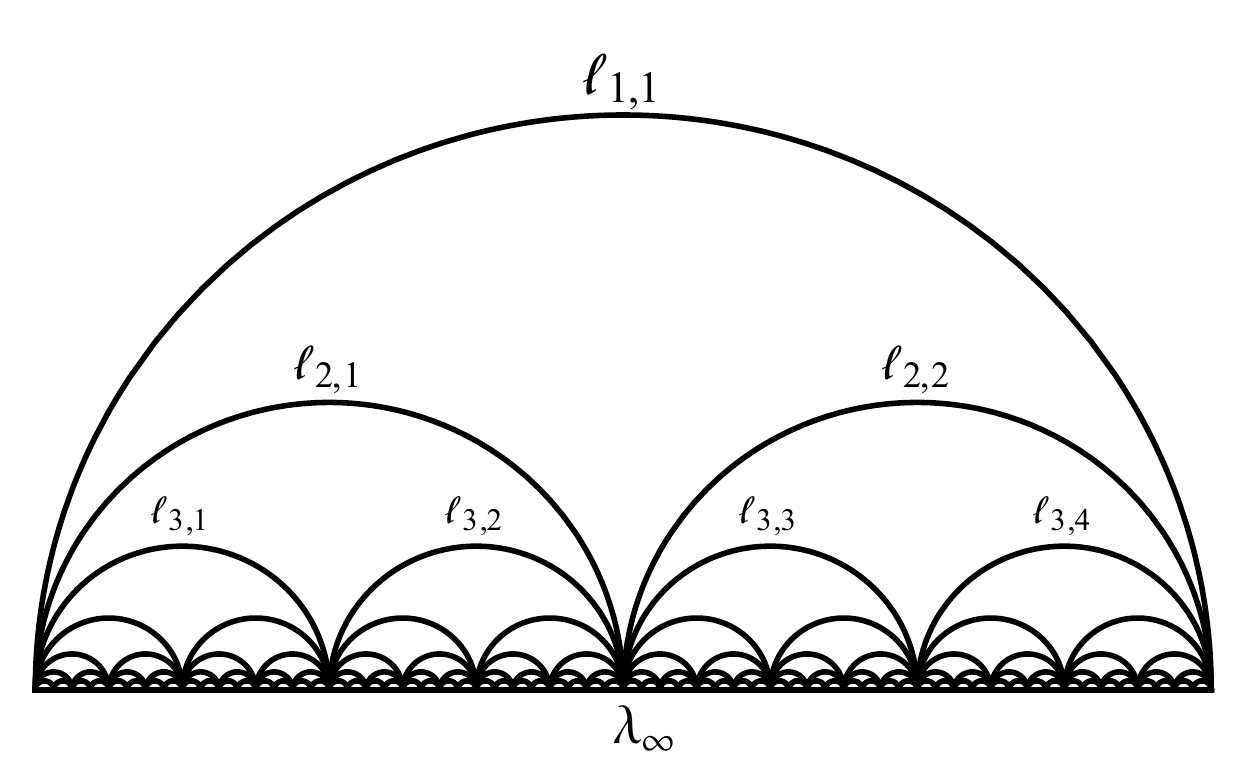}
\caption{\label{thespacedfig}The space $\bbd$ and paths $\ell_{n,j},\lambda_{\infty}$}
\end{figure}

For each $n\in\bbn$, let $E_n$ be the finite graph, which is the union of $B$ and all $\bbd(k,j)$ with $1\leq k\leq n$. The retractions $r_n:\bbd\to E_n$, which collapse $\bbd(k,j)$, $k>n$ vertically to the base-arc induce a homomorphism $\Phi:\pi_1(\bbd,d_0)\to \prod_{n\in\bbn}\pi_1(E_n,d_0)$, $\Phi([\alpha])=([r_1\circ\alpha],[r_2\circ\alpha],[r_3\circ\alpha],\dots)$. Since $\bbd$ is one-dimensional, and thus $\pi_1$-shape injective \cite{EK98onedimefund}, $\Phi$ is injective. Therefore, two paths $\alpha,\beta:\ui\to \bbd$ are path-homotopic if and only if $r_n\circ\alpha$ and $r_n\circ\beta$ are path-homotopic in $E_n$ for all $n\in\bbn$.

Recalling Example \ref{fandscex}, let $D$ denote the subgroup $F(\bbd,B)\leq\pi_1(\bbd,d_0)$. Note that $D$ consists of homotopy classes of loops based at $d_0$, which are finite concatenations of paths of the form $\ell_{n,j}$ or $\ell_{n,j}^{-}$. Let $d_{\infty}=[\ell_{1,1}\cdot\lambda_{\infty}^{-}]$; we consider the closure pair $(D,d_{\infty})$.

\begin{theorem}\cite[Theorem 4.13]{BFTestMap}\label{uplcharthm}
Suppose $H\leq \pionex$. If $p_H:\tXh\to X$ has the unique path-lifting property, then $H$ is $(D,d_{\infty})$-closed. The converse holds if $X$ is metrizable.
\end{theorem}

\subsection{Non-equivalence for non-normal subgroups}

We confirm the necessity of the test space $\bbd$ by showing that the well-definedness of transfinite $\Pi_1$-products relative to a general subgroup $H\leq \pionex$ does not guarantee the existence of a generalized covering map corresponding to $H$.

\begin{definition}
If $X$ is a space and $A\subseteq X$, let \[Nd(X,A)=\{[\alpha]\in\pionex\mid \alpha^{-1}(A)\text{ is nowhere dense or }\alpha\text{ is constant}\}.\]
\end{definition}
Note that $Nd(X,A)$ is a subgroup of $\pionex$ and $F(X,A)\leq Sc(X,A)\leq Nd(X,A)$. If $X$ is a one-dimensional metric space, reduction of paths takes place within the image of that path. Hence, if $\alpha$ is the reduced representative of a non-trivial element of $Nd(X,A)$, then $\alpha^{-1}(A)$ is nowhere dense.

\begin{lemma}\cite[Corollary 3.12]{BFTestMap}\label{cancellationsequence}
Let $\alpha:(\ui,0)\to (X,x_0)$ be a reduced path in a one-dimensional metric space $X$, $\gamma_n:\ui\to X$ be a null-sequence of reduced loops based at $\alpha(1)$, and $\eta_n$ be a reduced representative of $[\alpha\cdot \gamma_n\cdot \alpha^{-}]$. Then for every $0<t<1$, there exists an $N$ and $0<s<1$ such that $\eta_{N}|_{[0,s]}\equiv \alpha|_{[0,t]}$.
\end{lemma}

\begin{theorem}\label{ndiswclosed}
If $X$ is a one-dimensional metric space and $A\subseteq X$ is closed, then $Nd(X,A)$ is $(W,w_{\infty})$-closed.
\end{theorem}

\begin{proof}
Let $f:(\bbw,w_0)\to (X,x_0)$ be a map such that $f_{\#}(W)\leq Nd(X,A)$. Fix any $(a,b)\in\mci(\mcc)$. Since $\mci(\mcc)$ has dense order type, find \[(a_1,b_1)<(a_2,b_2)<(a_3,b_3)<\cdots <(a,b)\] in $\mci(\mcc)$ such that $a_n\to a$. Let $\alpha=f\circ\upsilon|_{[0,a]}$ and $\gamma_n=f\circ(\upsilon|_{[a_n,a]}^{-}\cdot \upsilon|_{[a_n,b_n]}\cdot\lambda_{[a_n,b_n]}^{-}\cdot\upsilon|_{[a_n,a]})$. If $I_n=(a_n,b_n)$, then $[\alpha\cdot\gamma_n\cdot \alpha^{-}]=f_{\#}(w_{I_n})\in Nd(X,A)$ for each $n\in\bbn$. Therefore, if $\eta_n$ is the reduced representative of $\alpha\cdot\gamma_n\cdot \alpha^{-}$ in $X$, then $\eta_{n}^{-1}(A)$ is nowhere dense. By Lemma \ref{cancellationsequence}, for every $0<t<1$, there exists an $N$ and $0<s<1$ such that $\eta_{N}|_{[0,s]}\equiv \alpha|_{[0,t]}$. Hence for each $0<t<1$, $\alpha|_{[0,t]}^{-1}(A)$ is nowhere dense. It follows that $\alpha^{-1}(A)$ is nowhere dense. 

Recall that $(a,b)$ was arbitrary, so by applying the previous paragraph to $(c_n,d_n)\in\mci(\mcc)$ with $d_n\to 1$, we see that for each $0<t<1$, $(f\circ\upsilon|_{[0,t]})^{-1}(A)$ is nowhere dense. It follows that $(f\circ\upsilon)^{-1}(A)$ is nowhere dense.

Again, fix $I=(a,b)\in\mci(\mcc)$ and define $\alpha=f\circ\upsilon|_{[0,a]}$. Let $\beta$ be the reduced representative of the path $f\circ\upsilon|_{I}$ and $\zeta$ be the reduced representative of $f\circ\lambda|_{I}$. We have already seen that $\alpha^{-1}(A)$ and $\beta^{-1}(A)$ are nowhere dense. Since $[\alpha\cdot\beta\cdot\zeta^{-}\cdot\alpha^{-}]=f_{\#}(w_I)\in Nd(X,A)$, if $\eta$ is the reduced representative of $\alpha\cdot\beta\cdot\zeta^{-}\cdot\alpha^{-}$, then $\eta^{-1}(A)$ is nowhere dense. Hence, if $\zeta^{-1}(A)$ contains an interval $(s,t)$, then $\gamma|_{[s,t]}$ must fully cancel in the reduction from $\alpha\cdot\beta\cdot\zeta^{-}\cdot\alpha^{-}$ to $\eta$. However, this would force the existence of an interval in either $\beta^{-1}(A)$ or $(\alpha^{-})^{-1}(A)$; a contradiction. We conclude that for each $I\in \mci(\mcc)$, the reduced representative $\zeta_I$ of $f\circ\lambda|_{I}$ has the property that $\zeta|_{I}^{-1}(A)$ is nowhere dense.

Define a map $g:\bbw\to X$ so that $g\circ\upsilon=f\circ\upsilon$ and for each $I\in\mci(\mcc)$, $g\circ\lambda|_{I}=\zeta_I$. Although $g\circ\lambda$ need not be a reduced path, the uniqueness of reduced representatives in path-homotopy classes ensures that $f$ is homotopic to $g$ as a based map. Thus $g_{\#}(w_I)=f_{\#}(w_I)$ for all $I\in\mci(\mcc)$ and $g_{\#}(w_{\infty})=f_{\#}(w_{\infty})$. By the previous paragraphs, we have that $(g\circ\upsilon)^{-1}(A)$ is nowhere dense and $(g\circ\lambda|_{I})^{-1}(A)$ is nowhere dense (treated as a subspace of $\overline{I}$) for each $I\in\mci(\mcc)$. It follows that $(g\circ\lambda)^{-1}(A)$ is nowhere dense. Finally, we see that the preimage of $A$ under $(g\circ\upsilon)\cdot (g\circ\lambda)^{-}$ is nowhere dense. Thus $f_{\#}(w_{\infty})=g_{\#}(w_{\infty})=[(g\circ\upsilon)\cdot (g\circ\lambda)^{-}]\in Nd(X,A)$.
\end{proof}

\begin{example}\label{examplenonnormal}
If $\bbd$ is the dyadic arc space with base arc $B$, then clearly $D\leq Nd(\bbd,B)$, however $d_{\infty}\notin Nd(\bbd,B)$ since $\ell_{1,1}\cdot\lambda_{\infty}^{-}$ is reduced and $\lambda_{\infty}^{-1}(B)=\ui$ has non-empty interior. Therefore, $Nd(\bbd,B)$ is a non-normal subgroup of $\pi_1(\bbd,d_0)$, which is $(W,w_{\infty})$-closed but not $(D,d_{\infty})$-closed. In particular, $p_{Nd(\bbd,B)}:\tX_{Nd(\bbd,B)}\to X$ does not have the unique path-lifting property. 
\end{example}

\begin{corollary}
$d_{\infty}\notin \cl_{W,w_{\infty}}(D)$.
\end{corollary}

\subsection{Equivalence for normal subgroups}

Let $f_{\mcc}:\ui\to \ui$ be the standard ternary Cantor map which is surjective, monotone, injective on $\mcc\backslash \bigcup_{(a,b)\in\mci(\mcc)}\{a,b\}$, and for each $I=(a,b)\in\mci(\mcc)$, there is a unique dyadic rational $u\in (0,1)$ such that $f_{\mcc}([a,b])=u$.

\begin{theorem}\label{mainthm}
If $H\leq \pionex$ is $(D,d_{\infty})$-closed, then $H$ is $(W,w_{\infty})$-closed. The converse holds if $H$ is normal. In particular, the closure operators $\cl_{D,d_{\infty}}$ and $\cl_{W,w_{\infty}}$ agree on normal subgroups.
\end{theorem}

\begin{proof}
The first statement is (1) of \cite[Proposition 7.6]{BFTestMap}. For the partial converse, suppose $H$ is a $(W,w_{\infty})$-closed normal subgroup of $\pionex$. Let $g:(\bbd,d_0)\to(X,x_0)$ be a map such that $g_{\#}(D)\leq H$. We will check that $g_{\#}(d_{\infty})\in H$.

Define a map $f:\bbw\to\bbd$ as follows: set $f(s,0)=(f_{\mcc}(s),0)$ where $f_{\mcc}$ is the ternary Cantor map. For each dyadic rational $\frac{2j-1}{2^n}\in (0,1)$, there is a unique $(a,b)\in\mci(\mcc)$ such that $f_{\mcc}([a,b])=\frac{2j-1}{2^n}$. Using this correspondence, we complete the definition of $f$ by setting $f\circ\upsilon|_{[a,b]}=\gamma_{n,j}$ where $\gamma_{n,j}=\ell_{n+1,2j-1}^{-}\cdot \ell_{n,j}\cdot \ell_{n+1,2j}^{-}$. The function $f$ is clearly well-defined. 

For continuity, it suffices to show $f$ is continuous at each $(c,0)\in \mcc\times \{0\}$. First, notice that if $I=(a,b)\in\mci(\mcc)$ has diameter $\frac{1}{3^n}$, then $f(C_I\cup ([a,b]\times \{0\}))$ has diameter $\frac{1}{2^{n-1}}$. Let $n\in\bbn$. By the continuity of $f_{\mcc}$, there exists an $m\in\bbn$ such that if $|x-c|<\frac{1}{3^m}$, then $|f_{\mcc}(x)-f_{\mcc}(c)|<\frac{1}{2^{n-1}}$. Now $U=\{(x,y)\in\bbw\mid |x-c|<\frac{1}{2(3^m)}\}$ is an open neighborhood of $(c,0)$ in $\bbw$ and if $(x,y)\in U$, then $d((x,y),(c,0))\leq \frac{1}{2(3^m)}+|x-c|<\frac{1}{3^m}$ in $\bbw$ and thus 
\begin{eqnarray*}
d(f(x,y),f(c,0)) &\leq& d(f(x,y),f(x,0))+d(f(x,0),f(c,0))\\
&=& d(f(x,y),f(x,0))+|f_{\mcc}(x)-f_{\mcc}(c)|\\
&<& \frac{1}{2^{n-1}}+\frac{1}{2^{n-1}}=\frac{1}{2^n}.
\end{eqnarray*}
This verifies the continuity of $f$.

Fix $I=(a,b)\in\mci(\mcc)$ so that $f_{\mcc}([a,b])=\frac{2j-1}{2^n}=u$. Note that $\alpha=f\circ\upsilon|_{[0,a]}$ is a path in $\bbd$ from $d_0$ to $(u,0)$. Since $u$ is a dyadic rational, there exists a path $\delta_u:\ui\to \bbd$ from $d_0$ to $(u,0)$ which is a finite concatenation of paths of the form $\ell_{n,j}$ or $\ell_{n,j}^{-}$.

Note that $f_{\#}(w_I)=[\alpha\cdot \ell_{n+1,2j-1}^{-}\cdot \ell_{n,j}\cdot \ell_{n+1,2j}^{-}\cdot \alpha^{-}]$. Let $k=[\delta_{u}\cdot \alpha^{-}]\in\pi_1(\bbd,d_0)$ and observe that 
\[kf_{\#}(w_I)k^{-1}=[\delta_{u}\cdot \ell_{n+1,2j-1}^{-}\cdot \ell_{n,j}\cdot \ell_{n+1,2j}^{-}\cdot \delta_{u}^{-}]\in D.\]
Since $g_{\#}(w_I)\in H$ and $H$ is normal, we have
\begin{eqnarray*}
(g\circ f)_{\#}(w_I) &=& g_{\#}(k)^{-1}g_{\#}(kf_{\#}(w_I)k^{-1})g_{\#}(k)\\
&\in & g_{\#}(k)^{-1}g_{\#}(D)g_{\#}(k)\\
& \leq & g_{\#}(k)^{-1}Hg_{\#}(k)\\
&=& H.
\end{eqnarray*}
Since $(g\circ f)_{\#}(W)\leq H$ and $H$ is $(W,w_{\infty})$-closed, we have $(g\circ f)_{\#}(w_{\infty})\in H$.

To finish the proof that $g_{\#}(d_{\infty})\in H$, it suffices to show that $f_{\#}(w_{\infty})=[(f\circ\upsilon)\cdot(f\circ\lambda)^{-}]$ is equal to $d_{\infty}=[\ell_{1,1}\cdot \lambda_{\infty}^{-}]$. Since $f\circ\lambda=\lambda_{\infty}\circ f_{\mcc}$, the path $f\circ\lambda$ is path-homotopic to $\lambda_{\infty}$. Therefore, it suffices to show that $f\circ\upsilon$ and $\ell_{1,1}$ are path-homotopic.

Recall the retraction $r_n:\bbd\to E_n$ and consider the projections $r_n\circ f\circ\upsilon$. Using only homotopies that delete constant subpaths, we have:
\begin{itemize}
\item[] $r_1\circ f\circ\upsilon\simeq \ell_{1,1}$
\item[] $r_2\circ f\circ\upsilon\simeq \ell_{2,1}\cdot \gamma_{1,1}\cdot \ell_{2,2}$
\item[] $r_3\circ f\circ\upsilon\simeq \ell_{3,1}\cdot \gamma_{2,1}\cdot\ell_{3,2}\cdot \gamma_{1,1}\cdot \ell_{3,3}\cdot\gamma_{2,2} \cdot \ell_{3,4}$
\item[] $\cdots$
\end{itemize}
Since each path on the right reduces to $\ell_{1,1}$, the projection $r_n\circ \ell_{1,1}$ is path-homotopic to $r_n\circ f\circ\upsilon$ for all $n$. By the $\pi_1$-shape injectivity of $\bbd$, $f\circ\upsilon$ is path-homotopic to $\ell_{1,1}$.

The final statement of the theorem now follows from \cite[Corollary 2.8]{BFTestMap}.
\end{proof}

Combining Theorem \ref{mainthm} with the characterizations in Proposition \ref{transfinitepathprodchar} and Theorem \ref{uplcharthm}, we obtain the following.

\begin{theorem}\label{mainthm1general}
If $p_H:\tX_H\to X$ has the unique path-lifting property, then $X$ has well-defined transfinite $\Pi_1$-products rel. $H$. The converse holds if $X$ is metrizable and $H$ is normal.
\end{theorem}

Theorem \ref{mainthm1} is the case $H=1$ of Theorem \ref{mainthm1general}.

\begin{corollary}
If $\lb\lb D\rb\rb$ is the normal closure of $D$ in $\pi_1(\bbd,d_0)$, then $d_{\infty}\in \cl_{W,w_{\infty}}(\lb\lb D\rb\rb)$.
\end{corollary}

\begin{proof}
$d_{\infty}\in \cl_{D,d_{\infty}}(D)\leq \cl_{D,d_{\infty}}(\lb\lb D\rb\rb)$ and, by Theorem \ref{mainthm}, $\cl_{D,d_{\infty}}(\lb\lb D\rb\rb)=\cl_{W,w_{\infty}}(\lb\lb D\rb\rb)$.
\end{proof}

\begin{corollary}
The normal closure $\lb\lb W\rb\rb$ of $W$ in $\pi_1(\bbw,w_0)$ is $(W,w_{\infty})$-dense in the sense that $\cl_{W,w_{\infty}}(\lb\lb W\rb\rb)=\pi_1(\bbw,w_0)$
\end{corollary}

\begin{proof}
By \cite[Lemma 4.1]{BFTestMap}, $D$ is $(D,d_{\infty})$-dense, i.e. $\cl_{D,d_{\infty}}(D)=\pi_1(\bbd,d_0)$. According to the proof of \cite[Proposition 7.6]{BFTestMap}, $\bbw$ may be identified as a subspace of $\bbd$ such that there is a retraction $r:\bbd\to \bbw$ satisfying $r_{\#}(D)\leq W$ and $r_{\#}(d_{\infty})=w_{\infty}$. Since $r_{\#}$ is a retraction of groups, it must be onto. Hence,
\begin{eqnarray*}
\pi_1(\bbw,w_0) &=& r_{\#}(\pi_1(\bbd,d_0))\\
&=& r_{\#}(\cl_{D,d_{\infty}}(D))\\
&\leq& \cl_{D,d_{\infty}}(r_{\#}(D))\\
&\leq & \cl_{D,d_{\infty}}(W)\\
&\leq & \cl_{D,d_{\infty}}(\lb\lb W\rb\rb)\\
&=& \cl_{W,w_{\infty}}(\lb\lb W\rb\rb)
\end{eqnarray*}
where the last equality follows from Theorem \ref{mainthm}.
\end{proof}

\begin{remark}
In combination with Lemma \ref{closurepropertieslemma}, the ``density" mentioned in the previous corollary implies that a normal subgroup $N\trianglelefteq \pionex$ is $(W,w_{\infty})$-closed if and only if for every map $f:(\bbw,w_0)\to (X,x_0)$ such that $f_{\#}(W)\leq N$, we have $f_{\#}(\pi_1(\bbw,w_0))\leq N$. 
\end{remark}

Let $\bbw\bba$ be the ``archipelago-like" space akin to $\bbd\bba$ in \cite{BFTestMap} where a 2-cell $D^{2}_I$ is attached to $\bbw$ along the loops $\upsilon_I\cdot \lambda_{I}^{-}$ for each $I\in\mci(\mcc)$ (See Figure \ref{waspacefig}). The following theorem is an alternative version of \cite[Theorem 6.1]{BFTestMap} where $\bbd$ is replaced with $\bbw$; we omit the proof since it is completely analogous.

\begin{theorem}
The following are equivalent for any path-connected metric space $X$:
\begin{enumerate}
\item $X$ admits a generalized universal covering,
\item every map $f:\bbw\to X$ such that $f_{\#}(W)=1$ induces the trivial homomorphism on $\pi_1$,
\item every map $f:\bbw\bba\to X$ induces the trivial homomorphism on $\pi_1$.
\end{enumerate}
\end{theorem}

\begin{figure}[H]
\centering \includegraphics[height=2.4in]{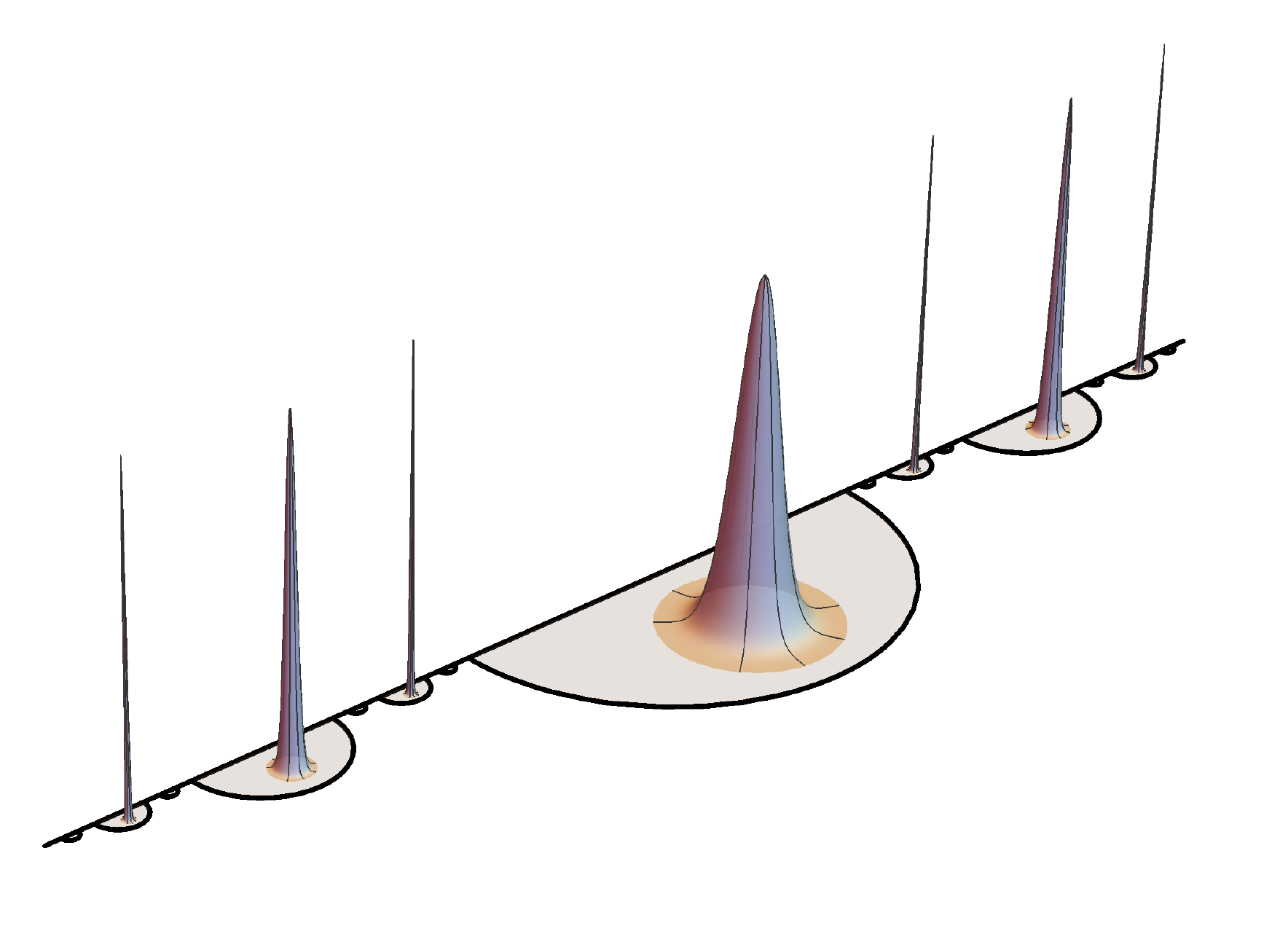}
\caption{\label{waspacefig}The space $\bbw\bba$}
\end{figure}

\begin{remark}
The space used to prove Theorem \ref{counterexamplespace} is highly non-compact. We recall that Eda's Problem asks if every homotopically Hausdorff Peano continuum admits a generalized universal covering. In light of the results in \cite{BrazScattered} and the current paper, we point out that Eda's Problem is equivalent to the following: If a Peano continuum $X$ has well-defined scattered $\Pi_1$-products, must it also have transfinite $\Pi_1$-products?
\end{remark}

\end{document}